\documentclass[12pt,reqno]{amsart}
\usepackage{amscd,amssymb,amsthm,euscript,graphicx,epsfig}
\usepackage[cmtip,arrow]{xy}
\usepackage{pb-diagram,pb-xy}
\newtheorem{thm}{Theorem}[section]
\newtheorem{prop}[thm]{Proposition}

\newtheorem{cor}[thm]{Corollary}

\newtheorem{definition}[thm]{Definition}
\theoremstyle{definition}
 \newtheorem{rmk}{Remark}[section]
\DeclareMathOperator{\rank}{\ensuremath{\mathop{rank}}}
\DeclareMathOperator{\spec}{\ensuremath{\mathop{spec}}}
\DeclareMathOperator{\mult}{\ensuremath{\mathop{mult}}}
\DeclareMathOperator{\interior}{\ensuremath{\mathop{interior}}}
\DeclareMathOperator{\spf}{\ensuremath{\mathop{Sf}}}
\DeclareMathOperator{\image}{\ensuremath{\mathop{image}}}

\newcommand{\fub}{\mathcal{F}_{R}^{sa}}
\newcommand{\h}{\mathcal{H}}

\newcommand{\B}{\mathcal{B}}
\newcommand{\E}{\mathcal{E}}

\newcommand{\C}{\mathbb{C}}
\newcommand{\R}{\mathbb{R}}
\newcommand{\F}{\mathcal{F}}

\newcommand{\U}{\mathcal{U}}

\newcommand{\K}{\mathcal{K}}
\newcommand{\x}{\times}
\newcommand{\iso}{\cong}
\newcommand{\dx}{\{D_{x}\}}
\newcommand{\Z}{\mathbb{Z}}
\newcommand{\cont}{\subseteq}
\newcommand{\minus}{\smallsetminus}

\DeclareMathOperator{\im}{\ensuremath{\mathop{image}}}

\begin{document}

\title{Spectral multiplicity and odd K-theory}

\author{Ronald G. Douglas}
\address{Department of Mathematics\\
Texas A\&M University\\
College Station, TX}
\email{rdouglas@math.tamu.edu}
\author{Jerome Kaminker}
\address{Department of Mathematical Sciences\\IUPUI\\Indianapolis, IN}
\curraddr{Department of Mathematics\\UC Davis\\Davis, CA}
\email{kaminker@math.ucdavis.edu}
\thanks{The research of the first author was supported in part by the
  National Science Foundation}  
\date{\today} 
\begin{abstract}
  In this paper we begin a study of families of unbounded self-adjoint
  Fredholm operators with compact resolvant.  The goal is to
  incorporate the information in the eigenspaces and eigenvalues of
  the operators, particularly the role that the multiplicity of
  eigenvalues plays, in obtaining topological invariants of the
  families.

\end{abstract}
\maketitle

\section{Introduction}

In the early sixties, K-theory, a generalized cohomology theory was
defined by Atiyah and Hirzebruch, \cite{Atiyah-Hirzebruch:1} based on a construction of
Grothendieck used earlier in algebraic geometry.  Following some
spectacular applications in topology, the development of K-theory was
intrinsically related the index theorem of Atiyah and Singer, \cite{atiyah-singerI68}.  One
result of this entanglement was the realization, by Atiyah, \cite{atiyah} and
J\"anich, \cite{Janich:1} of elements of the even group as homotopy classes of maps
into the Fredholm operators.  For the odd group, Atiyah and Singer
showed in \cite{Atiyah-Singer:6}
 that one could use homotopy classes of maps into the space of self-adjoint Fredholm
operators.  

Singer raised the question, \cite{Singer:1}, of describing elements in
the cohomology of the space of self-adjoint Fredholm operators in a
concrete way.  The generator of the first cohomology group can be
related to spectral flow and the eta invariant following the work of
Atiyah-Patodi-Singer, \cite{aps:1}, on the index
formula for certain elliptic boundary value problems.  This notion has
proved pivotal in a number of directions, including physics, where it
was used in the study of anomalies, \cite{Atiyah-Singer:PNAS}.  Let us be a little more precise.

In ~\cite{Atiyah-Singer:6} Atiyah and Singer established homotopy equivalences
between various realizations of the odd K-theory group.  The proofs
involved, among other things, a careful analysis of finite portions of
the spectrum of the operators.  However, the precise relationships of
some of these objects were left unresolved.  In subsequent years,
there has been some follow up on these ideas but analyzing families
has turned to other notions such as gerbes which can involve ancillary
structure.  Our goal in this paper is to return to the framework
introduced in ~\cite{Atiyah-Singer:6} and attempt to relate the odd classes
directly to the behavior of the self-adjoint operators.  

One can show that spectral flow is determined by just a knowledge of
the behavior of the eigenvalues of the family, along with their
multiplicities.  However, that is not true for the class in K-theory.
To overcome this defect, one must bring in the behavior of the
eigenspaces as well.  Hence, we seek to unravel this dependence and
understand how to obtain invariants similar to characteristic classes.
We believe it is likely that these relationships will have
applications to physics, such as ``higher anomalies'', and lead to a
study of ``higher'' spectral flow and index theory in general.  The
authors would like to thank Ryszard Nest for hospitality during a
visit to the University of Copenhagen where this project began.  We
would also like to thank Alan Carey who took part in initial
discussions on this topic and provided valuable insights.

This paper should be viewed as a first step in which some basic
structure is revealed and some results are obtained.  One goal here is
to formulate basic questions and frame critical issues which merit
further investigation.  Before providing an overview of our results we
need to introduce some definitions and notation.

We begin with more details on the space of self-adjoint operators.  It
is well known that the space of bounded self-adjoint Fredholm
operators on a separable Hilbert space $\h$, with the norm topology,
is a classifying space for odd K-theory.  That is, for any compact
Hausdorff space, $X$, one has
\begin{equation}
  K^{1}(X) \iso [X,\F ^{sa}].
\end{equation}
Here we are letting $\F^{sa}$ denote the component of the space of
bounded self-adjoint Fredholm operators which have both positive and negative essential
spectrum. We will also consider the subspace, $\F^{sa}_{0}$,
consisting of operators, $T$, with $\|T\| \leq 1$ and the essential
spectrum of $T$ equal to $\{\pm 1 \}$. 

In applications one is often provided with a family of unbounded
Fredholm operators parametrized by a locally path connected and
connected space $X$.  To be more precise we will consider the
following subset of unbounded self-adjoint Fredholm operators.  
\begin{definition}
\label{reg}
  The regular unbounded self-adjoint Fredholm operators, denoted $\fub$, consists of linear
  operators, $T$, which satisfy
  \begin{itemize}
  \item [i)] $T$ is closed and self-adjoint,
  \item [ii)]  $(I + T^{2})^{-1}$ is compact,
  \item [iii)] $T$ has infinitely many positive and infinitely many
    negative eigenvalues.
  \end{itemize}
\end{definition}

\begin{rmk}
We may also consider non-self-adjoint operators, $T$, satisfying the
condition that $T^{*}T \in \fub$.    Statements made about $\fub$ will hold with appropriate
modifications for such operators which we will denote by $\F_{R}$.
\end{rmk}
We shall study families of operators,
\begin{equation}
  D = \{ D_{x} \}: X \to \fub.
\end{equation}
Let $C_{b}(\R)$ denote the bounded continuous functions on $\R$ for
which the limits at $\pm \infty$ exist.  We consider families which
are continuous in the sense that the function
\begin{equation}
  \Theta : C_{b}(\R) \x X \to \B (\h)
\end{equation} defined by $\Theta(f,x) = f(D_{x})$  is norm
continuous. There is a topology on the set $\fub$ for which this holds.

With the Riesz topology, which is the one determined by the bounded
transform, $D \mapsto D(I+D^{2})^{-\frac{1}{2}}$, it has been shown by
L. Nicolaescu in \cite{Nicolaescu:2007} that the space of unbounded
self-adjoint Fredholms satisfying (i) and (iii) in Definition
\ref{reg}, but not necessarily (ii), provides a classifying space
for odd K-theory. There is a related result by M. Joachim in
\cite{joachim:2003} which states that those satisfying (ii) as well
also form a classifying space.  We will not need to make use of these results
in the present paper, but they will be relevant for future work.

The main examples will be families of Dirac operators on an
odd-dimensional manifold $M$ parametrized by a
compact space, $X$.  

A family  $\{ D_{x} \}$, as above, determines an element of $K^{1}(X)$.  It is
obtained by applying the function $\chi(x) = {x}(1+x^{2})^{-1/2}$ to each
operator $D_{x}$ to obtain the family of bounded self-adjoint Fredholm operators
\begin{equation}
  \label{eq:bdd}
  \{\tilde D_{x}\} = \{ \chi(D_{x}) \}.
\end{equation}  
Then each operator in the resulting family $\{\tilde D_{x} \}$ is a
bounded self-adjoint Fredholm operator and the homotopy class of the
family yields an element of $K^{1}(X)$.

The Chern character of such a  family, viewed in real cohomology,  has components only in  odd
degrees,
\begin{equation*}
  ch(\{\chi(D_{x})\}) \in \bigoplus_{i \geq 0}{H^{2i+1}(X,\R)}.
\end{equation*}

The class in $H^{1}(X,\R)$ corresponds to spectral flow, and the class
in $H^{3}(X,\R)$ is determined by the index gerbe, c.f.
~\cite{Lott:2002}.  One goal of the present work is to develop a
method that leads to a different description of these classes and the
higher dimensional classes which obstruct the triviality of the
K-theory class associated to the family.  These obstructions are to be
determined explicitly in terms of the spectrum and eigenspaces of the
operators in the family.  This is in a spirit similar to spectral flow
as we mentioned earlier.  As a first step, in the present paper we
will consider the role that the multiplicity of eigenvalues plays.

\section{The multiplicity of eigenvalues}

We will recall some basic definitions and facts that we will use.

\begin{prop}
  Let $D$ be an unbounded self-adjoint operator as above.  Let
  $\lambda$ be an eigenvalue of $D$.  Let $\delta > 0$ be such that
  there is no other eigenvalue in $[\lambda - \delta, \lambda +
  \delta]$.  Then the spectral projection onto the eigenspace for $\lambda$
  is
  \begin{equation}
    P_{\lambda}(D) = \dfrac{1}{2 \pi i} \int_{|z| = \delta} \dfrac{dz}{(z-D)}.
  \end{equation}
and the multiplicity of $\lambda$ is given by $m({\lambda},D) = \rank (P_{\lambda}(D))$.
\end{prop}

Now, consider a family of operators, $\{ D_{x} \}$.  We introduce the
following terminology.

\begin{definition}
  The graph (or spectral graph) of the family $\dx$, $\Gamma(\{ D_{x}\}) \cont X \x \R$, is
  \begin{equation}
    \Gamma(\{ D_{x}\}) = \{\ (y,\lambda)\ |\ \lambda \ \text{is an eigenvalue
      of}\  D_{y}\ \}.
  \end{equation}
\end{definition}
Note that  $\Gamma(\{ D_{x}\})$ is a closed subset of $X \x \R$.  When
the specific family is clear from the context, we will simply use
$\Gamma$ and simply call it the graph, dropping the term ``spectral''.

Both the spectral projection and the multiplicity of eigenvalues define functions on
the graph of the family.  We must consider continuity properties of
these functions.  

Let $(x,\lambda) \in \Gamma$ be a point in the
graph of the family.
\begin{definition}
  A canonical neighborhood of $(x,\lambda)$ is one of the form $V
  \x (\lambda-\delta , \lambda + \delta)$, where $x \in V$,  $\delta >
  0$, such that
  \begin{itemize}
  \item [a)] $\lambda$ is the only eigenvalue of $D_{x}$ in
    $(\lambda - \delta, \lambda + \delta)$,
  \item [b)] if $k = m(D_{x},\lambda)$, then,  for each $y \in V$,
    one has
    \begin{equation*}
      \sum_{(y,\mu) \in (V \x (\lambda - \delta, \lambda + \delta))
        \cap \Gamma}
m(D_{y},\mu) = k
    \end{equation*}
  \end{itemize}
for each $y \in V$.
\end{definition}

\begin{prop}
\label{standard}
Every point $(x,\lambda)$ in $\Gamma$ admits a canonical neighborhood,
$V \x (\lambda-\delta,\lambda+\delta)$ such that if $\lambda-\delta <
\lambda_{1}(y) \leq \ldots \leq \lambda_{k}(y) < \lambda+\delta$ are 
the eigenvalues of $D_{y}$ in the given interval,  then each
$\lambda_{j}(y)$ is continuous on $V$.
\end{prop}
\begin{proof}
  This will follow from a corresponding statement for bounded
  operators in \cite{boos-woj:book}.  However, we will need a precise
  form of this fact so we recall the steps. Let $f(t) =
  t/\sqrt{1+t^{2}}- \lambda/\sqrt{1+\lambda^{2}}$, and consider the
  family of bounded operators $\{ f(D_{x}) \}$.  Choose a $\delta > 0$
  so that there is no other eigenvalue of $D_{x}$ in $(\lambda-\delta
  , \lambda + \delta)$.  Assume $m(D_{x},\lambda) = k$.  Consider
  $f(D_{x})$ and $f(\delta)$ and apply \cite[p.  138]{boos-woj:book},
  to obtain a neighborhood $V$ of $x$ such that for $y \in V$ there
  are exactly $k$ eigenvalues of $f(D_{y})$ in $(f(\lambda-\delta) ,
  f(\lambda+\delta))$, which we will label $$-f(\delta) < \tilde \lambda_{0}(y) \leq
  \tilde \lambda_{1}(y) \leq \ldots \leq \tilde \lambda_{k}(y) <
  f(\delta),$$  Moreover,
  $|\tilde \lambda_{j}(y) - \tilde \lambda_{j}(y')| < \|f(D_{y}) -
  f(D_{y'}) \|$. Then $\lambda_{j}(y) = f^{-1}(\tilde
  \lambda_{j}(y))$, $V$,
  and $\delta$ yield the conclusion.
\end{proof}

As a corollary one obtains the following refinement.
\begin{prop}
\label{standard2}
Let $\lambda_{0}(x) < \ldots <\lambda_{n}(x)$ be a list of the
distinct eigenvalues of $\spec({D_{x}})$ which lie in a bounded interval
of $\R$.  Then there are disjoint canonical neighborhoods of each $(x,
\lambda_{j}(x))$, all with the same base $V$.
\end{prop}
\begin{proof}
  This follows easily from the method of proof of
  Proposition~(\ref{standard})
\end{proof}

Note that there can be no points of the graph between the standard
neighborhoods obtained.

It follows easily from this argument that the multiplicity
function will be lower semi-continuous.  The next result describes the
conditions under which it is actually continuous at a point
$(x,\lambda) \in \Gamma$.

\begin{prop}
\label{multiplicity}
  Let $U$ be a canonical neighborhood of $(x,\lambda) \in \Gamma$.  The following are equivalent.
  \begin{itemize}
  \item[i)] There is a positive integer ${k}$ so that the
    multiplicity function is  constantly equal to $k$ on $U$,
  \item[ii)] There is a $\delta > 0$ and a neighborhood $V$ of $x$ so that,
    for each $y \in V$,   $D_{y}$ has only one eigenvalue in the
interval $[\lambda - \delta, \lambda + \delta]$,
\item[iii)] The function associating the spectral projection to a
  point in the graph  is norm continuous on $
  \Gamma \cap (V \x [\lambda - \delta, \lambda + \delta]) $.
  \end{itemize}
\end{prop}

\begin{definition}
  The family $\{ D_{x} \}$ has constant multiplicity \emph{k} at
  $(x,\lambda)$ if it satisfies the conditions in Proposition
  \ref{multiplicity}. 

\end{definition}

We will next consider  criteria for the triviality of the K-theory
class associated to a family of self-adjoint operators.

\begin{prop} 
Let $\{D_{x}\} $ be a continuous family of self-adjoint operators.  The following are equivalent.
\label{trivial}
\begin{itemize}
\item[i)] The family defines the trivial element in $K^{1}(X)$,
\item[ii)] $\{D_{x} \}$ is homotopic to a family $\{D'_{x}\}$  for which there is a
  continuous function, $\sigma : X \to \R$, such that $\sigma(x)$ is
  not an eigenvalue of $D'_{x}$ for each $x$,
\item[iii)] $\{D_{x} \}$ is homotopic to a family $\{D'_{x}\}$ for
  which there is a norm continuous family of  projections
  $\{P'_{x}\}$ with range the sum of the eigenspaces for positive eigenvalues.
\item[iv)] There exists a spectral section for the family $ \dx $,
  in the sense of Melrose-Piazza,
  \cite{Melrose-Piazza:1998}. (i.e. there is a norm continuous family
  of projections which agree with the projections onto the positive
  eigenspaces outside of a closed interval, the interval itself depending
  continuously on $x$.)
 \end{itemize}
\end{prop}
\begin{proof}
This follows using the steps in the proof of Proposition 1 in Melrose-Piazza, \cite{Melrose-Piazza:1998}.
\end{proof}
\section{Spectral exhaustions and spectral flow}

In this section we will prove the existence and essential uniqueness
of spectral exhaustions.  Let $\{D_x\}$ be a continuous family of operators
parametrized by the compact space $X$, which we assume for now is a
simplicial complex.

\begin{definition}
  A spectral exhaustion for the family $\{ D_{x} \}$ is a family,
   of
  continuous functions $\mu_{n} : X \to \R$, indexed by $\Z$, satisfying
  \begin{itemize}
  \item[i)] $\mu_{n}(x)$ is an eigenvalue of $D_{x}$ for each $x$,
  \item[ii)]  $\{\mu_{n}(x) : n \in \Z \}$
  exhausts the spectrum of $D_{x}$ counting multiplicity, for each $x$,
\item [iii)] for each $x$ and for each $n \in \Z$, $\mu_{n}(x) \leq \mu_{n+1}(x)$.
  \end{itemize}
\end{definition}

\begin{rmk}
\label{remark1}
Note that, if the graphs of functions $\mu_{n}$ and $\mu_{n-1}$ are
disjoint and the parameter space $X$ is connected, then $\mu_{n}(x) >
\mu_{n-1}(x)$, for all $x$, so $\sigma(x) = \dfrac{1}{2}(\mu_{n}(x) -
\mu_{n-1}(x))$ satisfies condition (ii) of Proposition
\ref{trivial}.  Thus, the K-theory class of a family admitting a 
 spectral exhaustion with this property is trivial.
\end{rmk}
\begin{definition}
  An enumeration of the spectrum of an  operator $D \in \fub$ is a function $e_{D} : \Z \to
  \R$ mapping $\Z$ onto the spectrum of $D$ and satisfying
\begin{enumerate}
  \item [i)] if $\lambda$ is an eigenvalue of $D$ of multiplicity
    $k$, then there is an integer $N$ such that $\lambda = e_{D}(N) = e_{D}(N+1) =
    \ldots =e_{D}(N+k)$, and
  \item [ii)] $e_{D}(n) \leq e_{D}(n+1) $ for all $n$.
  \end{enumerate}
\end{definition}

Our goal in this section is to show that, if the spectral flow of the
family $\{D_x\}$ is zero, one can construct an enumeration of the the
spectrum of $D_{x}$, for each $x$, in such a way that the functions
$\mu_{n}(x) = e_{D_{x}}(n)$ are continuous.  Thus, we will obtain a
spectral exhaustion for $\{D_x\}$.

\begin{prop}
  Given an operator, $D$, an enumeration of the spectrum always exists
  and any two differ by translation by an integer.
\end{prop}
\begin{proof}
  Choose an eigenvalue, $\lambda$, of multiplicity $k$.  We set
  $e_{D}(0) = \lambda$ and $ e_{D}(-k+1) = \ldots =
  e_{D}(0) = \lambda$.  One can now uniquely extend this labeling to the rest
  of the spectrum.  It is easy to check that this process provides an
  enumeration of the spectrum of $D$.  Now suppose that $f_{D}$ is
  another one.  We will show that there is an $N $ such that $f_{D}(n
  + N) = e_{D}(n)$ for all $n$.  Let $\lambda$ be a point in the
  spectrum and   let $n_{0}$,
  $m_{0}$ be the largest integers so that $e_{D}(n_{0}) = \lambda =
  f_{D}(m_{0})$.  Let $N = m_{0} - n_{0}$.  Then it is easy to check
  that $e_{D}(n) = f_{D}(n+N)$ for all $n$.
\end{proof}

Note that the existence of an integer $n$ such that $e_{D}(n) =
f_{D}(n)$ is not sufficient to guarantee that $e_{D} = f_{D}$. However, if
there is an integer $N$ such that $e_{D}(N) = f_{D}(N)$  and
$e_{D}(N+1) > e_{D}(N) $, $f_{D}(N+1) > f_{D}(N) $, then it is the
case that $e_{D} = f_{D}$.

Fix $x \in X$ and let $\lambda$ be an eigenvalue of $D_{x}$.  Choose
an enumeration of the spectrum of $D_{x}$ satisfying
\begin{align*}
  e_{D_{x}}(0) & = \lambda \\
   e_{D_{x}}(1) &= \lambda' > \lambda.
\end{align*}

Find canonical neighborhoods
$W = V \x (\lambda -
  \delta, \lambda + \delta) $,
 $W' =  V \x (\lambda'-
  \delta', \lambda' + \delta') $ of $(x,\lambda)$ and $(x,\lambda')$
  respectively.

Let $\mu_{0}(y) = \max \{\lambda\  |\  \lambda \in \spec
(D_{y})\  \text{and}\  (y,\lambda) \in W \}$.  Similarly, let $\mu_{1}(y) = \min \{\lambda\  |\  \lambda \in \spec
(D_{y})\  \text{and}\  (y,\lambda) \in W' \}$.
\begin{prop}
  The functions $\mu_{0}$ and $\mu_{1}$ are continuous on $V$.
\end{prop}
\begin{proof}
 It will be sufficient to consider $\mu_{0}$, the case of $\mu_{1}$
 being similar.  Let $e_{0}(y)\leq \ldots \leq e_{k}(y)$ be the part
 of the spectrum of  $D_{y}$ in $ (\lambda-\delta,\lambda+\delta)$. Then
 $\mu_{0}(y) = e_{k}(y)$, and by the remark after
 Proposition~\ref{standard}, $\mu_{0}(y)$ is continuous.  
\end{proof}

Using $\mu_{0}$ and $\mu_{1}$ we define a spectral exhaustion over
$V$ by taking, for each $y \in V$, the unique (not just up to
translation) enumeration consistent with those choices.  Thus, we have
$\mu_{n}(y)$ defined for each integer $n$ and each $y \in V$.

\begin{prop}
  The functions $\mu_{n}$  are continuous on $V$ and, hence,
  $\{\mu_{n} \}$ is a spectral exhaustion over $V$.
\end{prop}
\begin{proof}
  Choose a point $y \in V$.  Taking a possibly smaller neighborhood
  $V'$ of $y$, we get $n+1$ disjoint canonical neigborhoods of the form $V' \x
  (\tilde \mu_{j}(y) - \delta_{j} ,\tilde \mu_{j}(y) + \delta_{j}) $,
  where $\tilde \mu_{j}(y)$ are the  eigenvalues of $D_{y}$
  from $\mu_{0}(y)$ to $\mu_{n}(y)$ listed multiply.  Now, for each $z\in V'$,
  $\mu_{n}(z)$ is in the  canonical neighborhood corresponding to the
  greatest real interval and it
  corresponds to one of the 
  eigenvalues $\lambda_{r}(z)$ in
  it.    We claim it must be the same $r$
  for each $z$ in $V'$.  To see this, let $N$ be the number of eigenvalues in
  the canonical neighborhoods below the top one and let $r$ be the
  index corresponding to $\mu_{n}(y)$.  Then $n = N +r$.  If we look
  at a point $z$ and  $\mu_{n}(z) = \lambda_{r'}(z)$, then we still must
  have $n=N+r'$, so that $r=r'$.  Thus, by Proposition \ref{standard},
  $\mu_{n}(z)$ varies continuously.
\end{proof}

Doing this construction in a neighborhood of each point $x \in X$, we
obtain a family of spectral exhaustions, each over an element of an
open cover, $\{ V_{i} \}$, where we may assume the open sets are
connected.  On the overlaps, any two exhaustions differ by an integer, so we obtain an
integer valued 1-cochain relative to $\{ V_{i} \}$ by taking the
difference of the partial exaustions, $\nu_{ij}= \mu_{0,i}|_{V_{i}\cap V_{j}} -
\mu_{0,j}|_{V_{i}\cap V_{j}} : V_{i}\cap V_{j} \to \Z$.  It is easily checked to be a
cocycle and its cohomology class in $\check H^{1}(X,\Z)$ will be
defined to be the spectral flow of the family, $\spf ( \{D_x\})$.  It
is straightforward to see that this definition agrees with other definitions of
spectral flow. (c.f.  \cite{boos-woj:book}).
\begin{thm} 
  A spectral
  exhaustion exists for the family $\dx$ if and only if the spectral
  flow of the family is zero, $\spf (\{D_x\}) = 0$.
\end{thm}
\begin{proof}
  If $\spf (\{D_x\}) = 0$ then the cocycle, which is defined with
  respect to the open cover $\{V_{i}\}$, is
    a coboundary, so that $\delta(\sigma) = \nu$ for some cochain
  $\sigma$. Then the 0-cochain with components $\mu_{n,i} - \sigma_{n,i}$ can be
  used to define global functions $\mu_{n}$.  These $\mu_{n}$'s
  provide the required exhaustion.

For the converse, if an exhaustion exists, this determines the choices
in constructing the cocycle representing $\spf (\{D_x\})$, and since
the locally defined exhaustion functions all piece together to yield
global functions, the class is equal to zero.
\end{proof}

\section{Families with spectrum of constant multiplicity}

In this section we will obtain the first results relating spectral
multiplicity to K-theory. Recall that  we assume that the
parameter space  is a finite simplicial
complex.  While this assumption is not always necessary, the  topology
issues that would arise with additional generality are not fundamental ones.

\begin{prop}
  Suppose that the family $\{ D_{x} \}$ has constant multiplicity at
  each point of a component, $\tilde X$, of $\Gamma$.  Then $pr_{1} :
  \tilde X \to X$ is a covering. 
\end{prop}
\begin{proof}
  We use Proposition \ref{multiplicity} (ii) which states that, for
  each $x \in X$ and each eigenvalue $\lambda$ of $D_{x}$ there is a
  neighborhood, $V$ and a $\delta > 0$ such that for each $y \in V$,
  $D_{y}$ has only one eigenvalue in the interval $[\lambda - \delta,
  \lambda + \delta]$.  Then the function, $\sigma_{x,\lambda} : V \to
  \R$, which sends $y$ to that eigenvalue, is continuous.  

  It then follows  that each
  component of $\Gamma(D_{x}) $ is a covering of $X$.
\end{proof}

We defined the spectral flow of a family $\{D_{x} \}$ to be a 
1-dimensional cohomology class,
\begin{equation*}
  \spf (\{D_{x}\}) \in \check H^{1}(X,\Z).
\end{equation*}
This class defines a homomorphism, for which we will use the same notation,
\begin{equation}
  \spf(\{D_{x}\}) : \pi_{1}(X) \to \Z. 
\end{equation}

The following is an easy consequence of the definitions.
\begin{prop}
  If a  component, $\tilde X$, of $\Gamma$  is a covering, $pr_{1}: \tilde X \to
  X$, then it corresponds to the
  homomorphism $\spf (\{ D_{x} \}) : \pi_{1}(X) \to \Z$. i.e. $\im
  ({pr_{1}}_{*}) = \ker (\spf(\{ D_{x} \}))$.  
\end{prop}

The next results give a criterion for the existence of
a spectral exhaustion with disjoint graphs.  
\begin{thm}
\label{cmtriv2}
  Let $ \{ D_{x} \} $ be a family with spectrum of constant
  multiplicity; that is, there exists an integer $k$ such that $m(D_{x}, \lambda) = k$, 
  for each $(x,\lambda)$ in $\Gamma$.  Assume that the spectral flow of the family is zero,
  \begin{equation*}
    \spf( \{ D_{x} \} ) = 0.
  \end{equation*}
Then a spectral exhaustion with functions having disjoint images, (except for repeated
functions due to multiplicity),  exists.
\end{thm}
\begin{proof}
  Since the multiplicity of the covering is constant, each component
  is a covering.
  Moreover, each of these coverings corresponds to the homomorphism
  \begin{equation*}
    sf(\{D_{x}\}): \pi_{1}(X,x_{0}) \to \Z,
  \end{equation*}
  given by spectral flow.  Thus, if the spectral flow of the family is
  zero, each of the coverings is a homeomorphism, so that $\Gamma
  \iso X \x \text{spec}(D_{x_{0}})$, for some point $x_{0} \in X$. Enumerate
  the spectrum of $D_{x_{0}}$ as $\{ \lambda_{n}(x_{0}) \}$ and let
  $\tilde X_{n}$ be the component of $\Gamma$ containing
  $\lambda_{n}(x_{0})$.  Then set $\mu_{n}(x) =
  pr_{2}\circ(pr_{1}|_{\tilde X_{n}})^{-1} $. These functions satisfy
  the requirements to be a spectral exhaustion, and their graphs,
  being the components of $\Gamma$, are disjoint.
\end{proof}

We obtain the following corollary from  Remark~\ref{remark1}.
\begin{cor}
\label{cmtriv}
  Let $ \{ D_{x} \} $ be a family with spectrum of constant
  multiplicity.  Assume that the spectral flow of the family is zero,
  \begin{equation*}
    \spf( \{ D_{x} \} ) = 0.
  \end{equation*}
Then the family $\dx$ is trivial in K-theory.
\end{cor}

It is also worth  noting the following result.
\begin{cor}
\label{compact}
  Suppose that some component of $\Gamma$ is compact.  Then the
  K-theory class of the family is trivial.
\end{cor}
\begin{proof}
  Let the component $\tilde X$ be compact. Then the number of
  sheets in the cover is the cardinality of $\pi_{1}(\tilde X)
  \iso \im (\spf(\{D_{x}\}))$, which must be finite.  However, this is a
  subgroup of $\Z$, so it will have to be zero.  Thus, the spectral
  flow of the family is zero and its class is trivial by Proposition
  \ref{cmtriv}.
\end{proof}

Finally, we consider how the hypothesis of constant multiplicity can
be replaced by an asymptotic version.

\begin{thm}
Let $\{D_{x}\}$ be a family with $\spf(\{D_{x}\})=0$
  Suppose that there exists an integer $N$ such that if $(x,\lambda)
  \in \Gamma$ and $\lambda > N$ then the family, $\{D_{x}\}$, has
  constant multiplicity at $(x,\lambda)$.  Then the class of the
  family is trivial in $K^{1}(X)$.
\end{thm}
\begin{proof}
  
  Let $\Gamma_{R} = \{ (x,\lambda) : \lambda > R \}$. We will show
  that there is an $R > N$ so that some component of $\Gamma_{R}$ is a
  covering of $X$.  If so, then as in \ref{cmtriv2}, $\spf(\{D_{x}\})
  = 0$ will imply that this component is compact and by Corollary
  \ref{compact} the K-theory class of the family will be trivial.

Thus, we must show that there is a path component, $\tilde X$, of $\Gamma$
which is contained in $\Gamma_{R}$ for some $R > N$. Since $\spf (\dx)
= 0$ a spectral exhaustion, $\mu_{n}$, exists. Let $\Gamma_{n} =
\image{\mu_{n}}$. For each $x \in X$ there exists an $n_{x}$ and a
  neighborhood of $x$, $U_{x}$, such that $\mu_{n}(y) > N+1$ for all
  $y \in U_{x}$.  Get a finite subcover, $U_{x_{1}}, \ldots,
  U_{x_{k}}$, and let $m = \max \{n_{x_{i}} \}$.  Then $\mu_{m}(x) >
  N+1$ for all $x \in X$.  This implies that the image of $\mu_{m}$ is
  a cover of $X$ and is compact and connected.
\end{proof}

\section{Families with bounded multiplicity}

In this section we will consider the question of when an element of
odd K-theory can be represented by a family with uniform bounded
multiplicity.  To this end let, for $n \geq 1$, $\fub (n)$ denote the operators with
multiplicity less than or equal to $n$.  Then $\fub(n) \cont
\fub(n+1)$ and we set $\fub(\infty) = \bigcup \fub(n)$. We do the same
for $\F_{R}$. Throughout,  $X$ will be a compact space.

Recall that in Atiyah-Singer, \cite{Atiyah-Singer:6}, the following
diagram was studied.
\begin{equation}
  \begin{diagram}
 \node{\F^{sa}_{0}} \node{\hat \F}
    \arrow{e,t}{\exp(i\cdot)} \arrow{w} \node{\U(I+\K)}
    \node{U_{\infty}} \arrow{w}\\
     \node{} \node{\hat\F_{n}} \arrow{n} \arrow{e,t}{\exp (i\cdot)}
    \node{\U(I+F_{n})} \arrow{n} \node{U_{n}} \arrow{w} \arrow{n}
  \end{diagram}
\end{equation}
Here, $\F^{sa}_{0}$ is the bounded self-adjoint Fredholms with
essential spectrum on both sides of the origin, while $\hat \F$ is
those operators with norm 1
and essential spectrum $\pm 1$.  Also, $\U(I+F_{n})$ is the unitary
operators of the form $I+K$, with $K$ of rank $n$, $\hat \F_{n} $ is
the operators in $\hat\F$ with finitely many eigenvalues in
$(-1,1)$ and for which $\exp(iT) \in U_{n}$.  The unlabeled arrows are
inclusions.  
The
Atiyah-Singer result shows that the composition of the maps on the top
row and their appropriate homotopy inverses provide a homotopy
equivalence which we shall denote by $\hat \chi : \F_{0}^{sa}
\to U_{\infty}$.  There is an obvious inclusion map of $\fub$ into
$\F_{0}^{sa}$ and hence into $U_{\infty}$.

To study the question of bounded multiplicity we make the following definition.

\begin{definition}
  Let $K^{1}_{(\infty)}(X)$ be the subset of $K^{1}(X)$ consisting of
  classes $[\alpha]$, $\alpha:X \to \fub$, such that there is an $n$
  and an $\alpha' \simeq \alpha$ with $\alpha':X \to \fub(n)$.  Let
  $K^{0}_{(\infty)}(X)$ be defined in an analogous way using $\F^{R}$.
\end{definition}

Note that the homotopy between $\alpha$ and $\alpha'$ is allowed to run
through all of $\F_{0}^{sa}$.

\begin{prop}
\label{natural}
  $K^{*}_{(\infty)}(X)$ is a natural subgroup of $K^{*}(X)$.
\end{prop}
\begin{proof}
  This subgroup is clearly preserved by induced homomorphisms and contains the
  identity element of $K^{*}(X)$.  Addition in $K^{1}(X)$ is induced
  by composition of operators which is homotopic to orthogonal direct
  sum.  Thus, the sum of classes represented by bounded
  multiplicity elements  is also represented by a family of bounded
  multiplicity.  Moreover, since  the inverse of an element
  given by a family $\alpha : X \to \F_{0}^{sa}$ is represented by $-\alpha$,
  this operation preserves the property of
  having bounded multiplicity.  Thus, the result follows.
\end{proof}

\begin{prop}
  $K^{*}_{(\infty)}(X)$ is mapped to itself under Bott periodicity.
\label{periodicity}
\end{prop}
\begin{proof}
  The Bott periodicity map is given by taking the product with the
  Bott element of $\tilde K^{0}(S^{2})$.  The product operation can be
  realized in the present context by letting each operator in the
  family act on the Hilbert space obtained by tensoring with the
  $L^{2}$ sections of the bundles representing the $K^{0}$ class.  If
  the bundle is trivial, then the multiplicity will by multiplied by
  its dimension.  If it is not trivial, then it is a summand of a
  trivial bundle and one can see that restricting to the image of the
  projection onto the sections of the bundle can only lower the
  multiplicity.  Note that in this setting, the operators in the family
  will commute with the projections onto the sections of the bundle.
\end{proof}

We will now make use of a construction which appears in a paper of
Mickelsson, \cite{Mickelsson:2006}.  It will be used to associate to a
map into the finite dimensional unitary group, $U_{n}$, an explicit
family of unbounded self-adjoint Fredholm operators with the
multiplicity of their spectrum uniformly bounded by $n$.

Let $U \in U_{n}$. Consider the operator
\begin{equation*}
  -i \frac{d}{dx} : C^{\infty}([0,1],\C^{n}) \to C^{\infty}([0,1],\C^{n})
\end{equation*}
with the boundary condition
\begin{equation*}
  \xi(1) = U \xi(0).
\end{equation*}

This yields a self-adjoint Fredholm operator on $L^{2}([0,1],\C^{n})$ which we will denote
$D_{U}$.  It is straightforward to compute the spectrum of $D_{U}$ and
the result is as follows.  Let $\{ z_{1},  \ldots, z_{n} \}$ be the
spectrum of $U$.  Let $\lambda_{j}$ satisfy $0 \leq \lambda_{j} < 1$
and $z_{j} = e^{2\pi i \lambda_{j}}$.  Then,
\begin{equation*}
  \spec(D_{U}) = \{ m + \lambda_{j}\ | \ m \in \Z, 1 \leq j \leq n \}.
\end{equation*}

The multiplicity of the eigenvalue $m + \lambda_{j}$ is the same as
that of the eigenvalue $z_{j}$ of $U$, and it follows that the multiplicity of the spectrum of $D_{U}$
is less than or equal to $n$.

Let $\mu_{n} : U_{n} \to \fub$ be defined by $\mu_{n}(U) = D_{U}$.  We
will refer to $\mu_{n}$ as the Mickelsson map.

\begin{prop}
\label{Mickelsson}
The Mickelsson map yields a map
\begin{equation}
  \mu : U_{\infty} = \bigcup_{n\geq 1} U_{n} \to \fub \to\F_{0}^{sa}  ,
\end{equation}
which induces an isomorphism on homotopy groups,
\begin{equation}
  \mu_{*} : \pi_{i}(U_{\infty}) \to \pi_{i}(\F_{0}^{sa}),
\end{equation}
for all $i$.
\end{prop}
\begin{proof} The first statement follows from the definitions while
  the second is a consequence of the facts that the Mickelsson map
  commutes with periodicity and the computation from \cite{Mickelsson:2006}
    that it is an isomorphism for $S^{3}$.
  
\end{proof}

Note that if $X$ is a compact space, then $\mu_{*} : [X,U_{\infty}] \to
[X,\F_{0}^{sa}]$ actually maps into $K^{1}_{\infty}(X)$.

These three propositions yield the following theorem.
\begin{thm} Let $X$ be a compact metric space.  Then one has
  $$K^{*}_{(\infty)}(X) = K^{*}(X). $$
\label{bounded}
\end{thm}
\begin{proof}
  It follows from Propositions \ref{natural} and \ref{periodicity}
  that $K^{*}_{(\infty)}(X)$ defines a cohomology theory on compact
  spaces with a 6-term exact sequence of the same type as that for
  $K^{*}(X)$.  The inclusion induces a map of 6-term sequences.
  Assume first that $X$ is a finite complex.  Then applying the
  cohomology theories to the sequence of skeletons,
  \begin{equation}
    X^{(k)} \to X^{(k+1)} \to \bigvee S^{(k+1)},
  \end{equation}
will yield the result by induction once one knows that it holds for
spheres.  However, for spheres the Mickelsson map composed with the
inclusion,
\begin{equation}
\begin{diagram}
  \node{\pi_{i}(U_{\infty})} \arrow{e,t}{\mu_{*}} \node{ \pi_{i}(\F^{sa}_{0}(\infty))} \arrow{e,t}{i} \node{\pi_{i}(\F^{sa}_{0})}
\end{diagram}
\end{equation}
agrees with the isomorphism from Atiyah-Singer,
\cite{Atiyah-Singer:6}.  Here, $\F^{sa}_{0}(\infty)$ denotes the
subset of $\F^{sa}_{0}$ homotopic to regular operators of bounded
multiplicity.  By Proposition \ref{Mickelsson}, $\mu_{*}$
is an isomorphism on spheres, hence so is the inclusion, $i$.  This
proves the result for finite complexes.  By expressing a compact
metric space as an inverse limit of finite complexes one obtains the
desired conclusion.

\end{proof}

Note that the same argument shows that the Mickelsson map is an
isomorphism.

This result has connections to the paper of Nicolaescu, \cite{Nicolaescu:2007},
in which the relation of $K^{1}(X)$ and homotopy classes of maps into
the space of unbounded self-adjoint Fredholm operators with essential
spectrum $\{\pm 1\}$, but possibly having some continuous spectrum, is
addressed.  From the vantage of this paper  Theorem \ref{bounded} shows that
every element in $K^{1}(X)$ is represented by a family of regular
unbounded self-adjoint Fredholm operators.  

As a consequence of this fact, one sees that any family is homotopic to a
family with bounded multiplicity.  One can estimate the bound on the
multiplicity in a rough way using the dimension of $X$.  It would be
desirable to get a refined estimate based on the topology of $X$.

\begin{definition}
  Let $\dx$ be a family on $X$.  The  {\em minimal multiplicity} of the
  family is the least integer $n$ such that $\dx \cong \dx'$ where
  $\dx'$ is a family with multiplicity bounded by $n$.  
\end{definition}
\begin{prop}
  Let $[\alpha] \in K^{1}(X)$ and suppose the dimension of $X$ is $k$.
   Then $[\alpha]$ is represented by a family $\dx$  with {\em minimal
   multiplicity} $ <
   [\frac{k+1}{2}]$, where $[x]$ denotes the largest integer less than
   or equal to $x$.
\end{prop}
\begin{proof}
   Suppose that $\alpha : X \to U_{N}$ is given.  Using the fibrations $U_{n-1}
   \to U_{n} \to S^{2n-1}$ one can inductively reduce the dimension of
   the unitary group to the least possible, which is
   $[\frac{k+1}{2}]$.  The result follows upon applying Theorem~\ref{bounded}. 
\end{proof}

\section{Multiplicity $\leq$ 2}

As a sample of how conditions on the multiplicity beyond assuming
constancy can be used to study the K-theory class of a family, we will
consider the case when the multiplicity is less than or equal to 2.
We will also assume that the spectral flow of the family is zero.  The
main result of this section is that, if we assume that the space $X$
has no torsion in cohomology, then such a family is trivial in
K-theory if a certain 3-dimensional cohomology class vanishes.  Thus, the
index gerbe will be zero also.

Let us assume that we have a family with $\mult(\{D_x\}) \leq 2$  and
recall the standing assumption that the parameter space $X$ is a
connected finite simplicial complex.  Assume $\spf (\dx) = 0$ and let
$\{\mu_{n}\}$ be an exhaustion.  Our procedure will be to 
deform the family inductively over
k-skeletons for increasing $k$,  
so that the exhaustion for the deformed family, $\{\tilde
\mu_{n}\}$, has $\tilde \mu_{0}(x) < \tilde \mu_{1}(x)$ for all $x$.  The
triviality will then follow from Proposition~\ref{trivial}.  

Let $C_{i,i+1} = \{\ x\ |\ \mu_{i}(x) = \mu_{i+1}(x) \}$, for any $i
\in \Z$.  Since $\mult(\dx) \leq 2$ we have $C_{-1,0},\  C_{0,1}$ and
$C_{1,2}$ disjoint closed sets. Let $W_{i,i+1}, \ i= -1, 0, 1$, be disjoint
open neighborhoods of $C_{i,i+1}$.  We assume the triangulation of $X$
so fine that any closed simplex which meets $C_{i,i+1}$ is contained in
$W_{i,i+1}$.  Thus, there are a finite number of simplices,
$\sigma_{l}$, such that
\begin{equation*}
  C_{i,i+1} \cont \interior (\bigcup_{1}^{n}\sigma_{l}) \cont \bigcup_{1}^{n}\sigma_{l} \cont W_{i,i+1}.
\end{equation*}

Our procedure for deforming a family involves successive application
of certain types of ``moves''.  The first is a preliminary flattening
process which allows one to control the geometry of the sets over
which the family has eigenvalues of multiplicity 2.  
We will state things for $C_{0,1}$ to simplify notation, but all
results hold for $C_{i,i+1}$ with the appropriate modifications. 

\begin{prop}[Flattening]
  \label{flattening}
  
  Let $K$ be a closed subset of $C_{0,1}$  and let $W_{1}, W_{2}$ be 
  open sets with compact closures satisfying $K \cont W_{1} \cont \bar
  W_{1} \cont W_{2} \cont \bar W_{2} \cont W_{0,1}$.  Assume further
  that $K = C_{0,1} \cap W_{2}$.   
  Then there exists a family $\{\tilde D_x\}$ with
  associated exhaustion $\{\tilde \mu_{n} \}$, which  satisfies
  \begin{itemize}
  \item[i)] $K \cont W_{2} \cap \tilde C_{0,1} = \bar W_{1}$, where $\tilde C_{0,1} = \{\
    x\ |\ \tilde \mu_{0}(x) =
    \tilde \mu_{1}(x) \}$,
  \item [ii)] $\tilde D_{x} = D_{x}$ for $x \in X \minus W_{2}$, and
  \item [iii)] $\{\tilde D_x\} \simeq \{ D_x\} $. 
  \end{itemize}

\end{prop}
\begin{proof}
   Let $\phi: X \to [0,1]$ be a function satisfying
  \begin{equation*}
    \phi(x) = \begin{cases}
         0& \text{for}\ \  x \in X \minus W_{2}\\
         1& \text{for}\ \  x \in \bar W_{1}
              \end{cases}     
  \end{equation*}

Define
\begin{equation*}
  \tilde D_{x,t} = D_{x} + t\phi(x)(h_{x}(D_{x}) - D_{x},
\end{equation*}
where $h_{x}: \R \to \C$ is a continuous function satisfying
\begin{equation*}
    h_{x}(t) = \begin{cases}
         t& \text{for}\ \  t \leq \mu_{-1}(x)\ \  \text{or}\ \ t \geq
         \mu_{1}(x) \\
         
         \mu_{1}(x)& \text{for}\ \  \mu_{0}(x) \leq t \leq \mu_{1}(x) \\ 
         \lambda_{x}(t)& \text{for}\ \  \mu_{-1}(x) \leq t \leq \mu_{0}(x) \\
          \end{cases},     
  \end{equation*}
where $\lambda_{x}(t)$ is the linear function with graph connecting
$(\mu_{-1}(x), \mu_{-1}(x))$ to $(\mu_{0}(x), \mu_{1}(x))$

Letting $\tilde D_{x} = \tilde D_{x,1}$, with associated exhaustion
$\tilde \mu_{n}(x)$, one checks that $K \cont W_{2} \cap \tilde C_{0,1} = \bar W_{1}$
and that the conclusions of the proposition hold for the family
$\{ \tilde D_{x} \}$.
\end{proof}

Thus, the preceeding deformation allows one to determine the set precisely,
($\bar W_{1}$ above), on which multiplicity of $(x, \tilde \mu_{0}(x))$
is 2.  Next, we will modify the family over neighborhoods of these
sets.

Let $X^{(0)}$ be the $0$-skeleton of the
parameter space $X$.  The first step will be to deform the family
$\{D_x\}$ on a neighborhood of $X^{(0)}$.
\begin{prop}
\label{vertices}
  Let $y \in X^{(0)}$.  Then there exists  a contractible neighborhood $V$  with $V
  \cap X^{(0)} = \{y\}$ and a family $\{\tilde D_x\}$ satisfying
  \begin{itemize}
  \item [i)] $\tilde D_{x} = D_{x}$ for $x \in X \minus V$,
  \item [ii)]$\{\tilde D_{x} \} \simeq \{ D_{x} \}$, and
  \item [iii)] There is a neighborhood $W$ of $y$ such that $\tilde \mu_{0}(x) < \tilde \mu_{1}(x)$ for $x \in W \cont \bar W \cont V$.
  \end{itemize}
\end{prop}
\begin{proof}
  If $\mu_{0}(y) < \mu_{1}(y)$ then this uniquely will hold in a neighborhood of
  $y$ and the original family will satisfy conditions (i)-(iii).  If,
  on the other hand, $y \in C_{0,1}$, so that $\mu_{-1}(y) < \mu_{0}(y) = \mu_{1}(y) < \mu_{2}(y)$,
  then we apply Proposition \ref{flattening} to obtain a 
  contractible neigborhood with compact closure, 
  $V$, of $y$ on which $\mu_{0}(x) = \mu_{1}(x)$ for $x \in V$.    

Let $\E = \{(x,v) \in V \x \h\
|\ v \ \text{is in the span of the eigenvectors for}\ \ \mu_{0}(x)\ \text{and} \ \mu_{1}(x)
\}$.  Then $\E \to X$ is a 2-dimensional vector bundle on some,
possibly smaller, neighborhood of $y$ which we continue to call $V$.
Since $V$ is contractible, the bundle is trivial.
Thus, there exists a framing $\{ \sigma_{0},
  \sigma_{1}\}$, where $\sigma_{j}(x)$ is an eigenvector for
  $\mu_{j}(x)$, for $j=0,1$.  Shrink $V$ to get $W \cont \bar W \cont V$
  and, using a bump function $ \phi$,  we  extend $\sigma_{i}$
  to all of $X$.  Let $\alpha(t,x) = t
  \frac{\mu_{1}(x)+\mu_{2}(x)}{2}$, and set $\tilde D_{x,t} = D_{x} +\alpha(t,x)
  P_{\sigma_{1}(x)}$, where $P_{\sigma_{1}(x)}$ is the orthogonal
  projection onto the subspace spanned by $\sigma_{1}(x)$. The family $\{\tilde D_{x,1}\}$
  satisfies the requirements of the proposition.

We repeat this construction, with the obvious modifications, for
vertices in $C_{-1,0}$.  Alternatively, one may observe that it is
possible to apply this method to the vertices in both $C_{-1,0}$ and
$C_{0,1}$ simultaneously.
\end{proof}

Thus, we have deformed our family so that $\mu_{0}(x) < \mu_{1}(x)$ on
a neighborhood of the 0-skeleton.  We will now proceed inductively to
extend this separation to all of $X$. 

It is worth noting that the previous deformations and all future ones
have the property that they are monotone, in the sense that
$\mu_{i}(x) \leq \tilde \mu_{i}(x)$ for all $i$ and $x \in X$.

From now on, to simplify notation, we will 
drop the tilde  and rename the  family resulting from a deformation as $\{D_{x}\}$.  

Assume that our family has the property that $\mu_{0}(x) < \mu_{1}(x)$
on a neighborhood, $W$, of the (k-1)-skeleton.  Let $\C^{2}\to \E \to
W$ be the 2-dimensional vector bundle whose fiber over the point $x$
is the span of the eigenspaces for $\mu_{0}(x)$ and $\mu_{1}(x)$.  For
$x \in W$, define $\sigma_{0}(x) $ to be the orthogonal projection
onto the eigenspace for $\mu_{0}(x)$.  This defines a continuous field
of projections over $W$. The next result shows that one can obtain the
desired deformation of $\dx$ if the field $\sigma_{0}$ can be extended
to the k-skeleton.  We will construct  the deformation simplex by simplex.

\begin{prop}
  \label{edges}
Let $\dx$ be a family of operators with $\mu_{0}(x) < \mu_{1}(x)$ on a
neighborhood, $W$, of the (k-1)-skeleton.  Let $\Delta^{k}$ be a
k-simplex.  Let $\ '\Delta^{k}$ be a smaller k-simplex with $C_{0,1}\cap
\Delta^{k} \cont \interior('\Delta^{k}) \cont\  '\Delta^{k} \cont
\interior(\Delta^{k})$.  If $\sigma_{0}|_{\partial ('\Delta ^{k})}$
extends to a field of projections onto the eigenspaces for
$\mu_{0}(x)$ over $\Delta^{k}$, then there are neighborhoods, $V$ and
$W'$, with 
$X^{(k-1)} \cup \Delta^{k} \cont W' \cont \bar W' \cont V$ and a family $\{\tilde  D_x\}$ satisfying

\begin{itemize}
  \item [i)] $\tilde D_{x} = D_{x}$ for $x \in X \minus V$,
  \item [ii)] $\{\tilde D_{x} \} \simeq \{ D_{x} \}$, and
  \item [iii)] $\tilde \mu_{0}(x) < \tilde \mu_{1}(x)$ for $x \in W' $.

\end{itemize}
\end{prop}
\begin{proof}

Let $\phi(x)$ be a bump function which is 1 on $W'$ and 0 on
$X\setminus V$.  Let $\tilde D_{x,t} = D_{x} +
t\phi(x)(\frac{\mu_{1}(x) + \mu_{2}(x)}{2})\sigma_{0}(x)$.  Then it is
straightforward to check that the family $D_{x,1}$ satisfies
conditions (i) -- (iii).  
\end{proof}

In order to use this construction to deform a family which is
separated over a neighborhood of $X^{(k-1)}$,
we apply Proposition~\ref{flattening}.  For this, note that $C_{0,1}$
is contained in the interior of k-simplices.  Consider one such
simplex and find a sub-simplex with parallel sides which contains its
intersection with $C_{0,1}$ in its interior.  We may assume that the
boundary of the smaller simplex is contained in the open set over
which $\mu_{0}(x) < \mu_{1}(x)$.  Next one applies the flattening
lemma to obtain a new family which has the smaller simplex as exactly
$C_{0,1} \cap \Delta^{k}$.  The projection field is defined on $\partial ('\Delta^{k})$.
Thus, if one can alway extend these projection fields from the boundary
of a k-simplex to the interior, then we can accomplish the deformation
of the family to one for which $\mu_{0}(x) < \mu_{1}(x)$ over a
neighborhood of the k-skeleton.  

Note that, since $\Delta^{k}$ is
contractible, the bundle $\E \to X$ is trivial over $\Delta^{k}$.
Thus, the existence of the extension is equivalent to the map $\sigma_{0}:
\partial \Delta^{k} \to Gr_{2}(\C^{2})$ being null-homotopic.  Here
$Gr_{2}(\C^{2})$ is the Grassmannian of lines in $\C^{2}$, which is
homeomorphic to $S^{2}$.  Since $\pi_{0}(S^{2}) = \pi_{1}(S^{2}) = 0$,
one can always obtain a deformed family for which $\mu_{0}(x) <
\mu_{1}(x)$ on a neighborhood of the 2-skeleton, $X^{(2)}$.  However,
since $\pi_{2}(S^{2}) = \Z$, it is not clear that one can proceed.  We
shall address this in the application below.


Putting these facts together we briefly present a sample of the type of result
obtainable using these methods.  Note that the result below makes the
case that the multiplicity of eigenvalues has a strong effect on the
classification of families of operators.
\begin{thm}
\label{result}
  Let $\dx$ be a family on $S^{2n+1}$ with $n \geq 2$.  If the
  multiplicity is less than or equal to 2, then the family is
  rationally trivial
  in K-theory.
\end{thm}
\begin{proof}
  We sketch the argument. Note that given a separation over an open
  set, we get a continuous eigenprojection field over that set. Thus,
  as above, we may obtain an eigenprojection field over a neighborhood
  of the 2-skeleton.  For the induction step, let $k < 2n+1$ and assume
  that in a neighborhood of the $k$-skeleton we have $\mu_{-1}(x) <
  \mu_{0}(x) < \mu_{1}(x)$.  Thus there is a rank 1 eigenprojection
  field over this neighborhood.  We will deform the family so that
  this property holds over the $k$-skeleton, hence on a neighborhood
  of it.  

  We now try to extend the eigenprojection field over the 3-skeleton.
  Proceeding as above by flattening and deforming, we obtain an eigenprojection
  field over the boundaries of 3-simplices.  We take a slightly
  different approach to complete the deformation process.  We first try to extend to
  simply a general projection field.  This can be done if the
  eigenprojection fields defined on the boundaries of the 3-simplices,
  $S^{2} \to Gr_{1}(\h)$, are null-homotopic.  Since $Gr_{1}(\h) = \C
  P ^{\infty}$, this is not automatic.  However,  obstruction theory applies and
  there is a class in $H^{3}(S^{2n+1},\pi_{2}(Gr_{1}(\h))$ which must vanish
  in order for the extension to exist (after going back and redefining
  over lower skeleta).  Since we are working with a sphere of
  dimension greater than 5, this group vanishes and the extension
  exists.

  We now proceed by induction.  We have a field over all of $S^{2n+1}$ which
  is an eigenprojection field over a neighborhood of the $k$-skeleton.
  We next try to push this field down to be an eigenprojection field
  over the $(k+1)$-skeleton.  We consider the set of $(k+1)$-simplices
  and their boundaries.  There are two cases.  If a $(k+1)$-simplex
  doesn't contain any singular points, then we have two projection
  fields over it--an eigenprojection field and the general one we just
  obtained.  Any two are homotopic relative to the boundary because
  only the second homotopy group of $Gr_{1}(\h)$ is non-zero.  We will
  assign 0 to such a simplex.  For the others, using the flattening
  procedure described above, we will get an element in the relative
  homotopy group, $\pi_{k}(Gr_{1}(\h),Gr_{1}(\C^{2})) \iso
  \pi_{k-1}(Gr_{1}(\C^{2})) = \pi_{k-1}(S^{2})$.  This will define an
  obstruction class in $H^{k+1}(S^{2n+1},(S^{2n+1})^{(k)};\pi_{k-1}(S^{2}))$.

  We note that, for any complex $X$, one has $H^{j}(X,X^{(k)}) =0$ if
  $j \leq k$, and $H^{j}(X,X^{k}) \iso H^{j}(X)$ if $j > k$.  Thus,
  the only group which can be non-zero is
  $H^{2n+1}(S^{2n+1},(S^{2n+1})^{(2n)}); \pi_{2n-1}(S^{2}))$, a
  torsion group. For $k<2n+1$ we extend the restriction of the
  eigenprojection field to the $(k-1)$-skeleton to the $(k+1)$-skeleton
  and use this to deform the family so it is separated there.  For the
  case of the top dimension the following procedure takes care of this
  obstacle.

  We have a possibly non-trivial, obstruction class which would belong
  to $H^{2n+1}(S^{2n+1};\pi_{2n}(Gr_{1}(\h),Gr_{1}(\C^{2})))$.  Since
  $n \geq 2$, $\pi_{2n}(Gr_{1}(\h),Gr_{1}(\C^{2})) \iso
  \pi_{2n-1}(S^{2})$ is a finite group, say of order $N$.  We will
  show that $N\dx$ is trivial, and hence $\dx$ is rationally trivial.
  Let $\Delta$ be an $2n+1$-simplex which meets $C_{0,1}$ in its
  interior and let $\Delta'$ be a sub-simplex obtained by flattening.
  Thus, we may assume that the family has multiplicity two at each
  point of $\Delta'$ and there is a rank 1 eigenprojection field along
  the boundary.  Now, consider the family $N\dx$.  We trivialize each
  2-dimensional eigenbundle separately and get a map $$c(\Delta):
  (\Delta', \partial \Delta') \to (Gr_{1}(\h) \x \ldots \x Gr_{1}(\h),
  Gr_{1}(\C^{2}) \x \ldots \x Gr_{1}(\C^{2}))$$ defining a class in
  $\pi_{2n}(Gr_{1}(\h), Gr_{1}(\C^{2})) \oplus \ldots \oplus
  \pi_{2n}(Gr_{1}(\h), Gr_{1}(\C^{2}))$.  The map induced by addition
  on homotopy groups sends the class of this map to zero in
  $\pi_{2n}(Gr_{1}(\h), Gr_{1}(\C^{2}))$. From the commutative
  diagram,
  \begin{equation*}
    \begin{diagram}
      \node{\pi_{2n}(Gr_{1}(\h), Gr_{1}(\C^{2})) \oplus \ldots \oplus \pi_{2n}(Gr_{1}(\h), Gr_{1}(\C^{2}))}\arrow{s,l}{+}\arrow[2]{e,t}{inc_{*}}\node{}\node{\pi_{2n}(Gr_{1}(\h), Gr_{N}(\C^{2N}))}\\
      \node{\pi_{2n}(Gr_{1}(\h), Gr_{1}(\C^{2}))}\arrow{ene,b}{inc_{*}}\node{}
    \end{diagram}
  \end{equation*}
  we see that we may deform the map $c(\Delta):  \Delta' \to
  Gr_{1}(\h)$ to one mapping into $Gr_{N}(\C^{2N})$.  Using the same deformation as in the
  rank 1 case, we obtain a new family for which $\mu_{0}(x) <
  \mu_{1}(x)$ for all $x \in \Delta'$.  Doing this process over each
  $n$-simplex yields the required trivial family.  
\end{proof}

\section{Concluding remarks}

The intention of the present paper was to begin a study of the manner
in which the variation of the eigenvalues and eigenspaces of a family
of self-adjoint Fredholm operators effects the K-theory class of the
family.  While we showed that the behavior of the multiplicity
function can effect the topology of the family, there is much yet to
be resolved.  Some questions which seem essential to making further
progress are listed below.

\begin{itemize}

\item  How is the 3-dimensional integral cohomology class which arises
  when applying obstruction theory related to the index gerbe, c.f. Lott,
  \cite{Lott:2002}.  If they determine each other, can one obtain all the
  components of the Chern character of the family using these methods?
\item Suppose $[\alpha] \in K^{1}(X)$ and there is an $\alpha'$ with
  $\alpha \simeq \alpha'$ and with
  the multiplicity of $\alpha'$ bounded by $n$.  Let $\mathcal M ([\alpha])$ be
  the least such $n$.  If $[\alpha] \neq 0$ then  $\mathcal
  M([\alpha]) > 1$.  How are the topological invariants of $[\alpha]$
  related to $\mathcal M ([\alpha])$? 
\item  The equivalence relation generated by the ``moves'' we are
  using to deform the families is possibly stronger than
  homotopy.  Is one obtaining a more refined type of K-theory in this way?
\item It would be interesting to know in what sense the K-theory class
  of a family is determined by a finite part of the spectrum.  To be
  more precise, suppose $\dx$ is a family with multiplicity bounded by
   $n$.  Is there an integer $\mathcal N (n)$ so that  the part of
   the graph of the family, $\bigcup_{|k| < \mathcal N (n)}
   \mu_{k}(X)$, along with the corresponding eigenspaces, determines
   whether the family is trivial (or rationally trivial) in K-theory? 
 \item Although various partial results similar to those in the last
   section are known to the authors, the appropriate general statement
   has not yet been obtained.  We expect that the following will hold.
   Assume the parameter space of the family, $X$, is an n-dimensional
   finite complex and the spectral flow of the family is zero.
   Further, suppose there is an element of the exhaustion, $\mu_{k}$
   such that the multiplicity at any point of $\mu_{k}(X)$ is $N$ or
   $N+1$, where $N > n$.  Then if the 3-dimensional obstruction
   obtained above is zero, the family is rationally trivial in
   K-theory.

\end{itemize}
  

\providecommand{\bysame}{\leavevmode\hbox to3em{\hrulefill}\thinspace}
\providecommand{\MR}{\relax\ifhmode\unskip\space\fi MR }
\providecommand{\MRhref}[2]{%
  \href{http://www.ams.org/mathscinet-getitem?mr=#1}{#2}
}
\providecommand{\href}[2]{#2}

\end{document}